\documentclass[11pt]{amsart}
\usepackage{}
\usepackage{amssymb}
\theoremstyle{plain}
\newtheorem{Thm}{Theorem}

\begin{document}

\title[warped ancient Ricci solutions]
{Liouville theorem for warped ancient Ricci solutions}

\author{Li Ma, Anqiang Zhu}

\address{Department of mathematics \\
Henan Normal university \\
Xinxiang, 453007 \\
China}
\email{nuslma@gmail.com}

\address{Department of mathematics \\
Wuhan university \\
Wuhan, 430072 \\
China} \email{aqzhu.math@whu.edu.cn}

\thanks{The research is partially supported by the National Natural Science
Foundation of China (N0.11271111) and the National Natural Science
Foundation of China (NO.11301400)}

\begin{abstract}
In this note, we answer affirmatively the question if a warped product of a compact manifold with a line as an ancient solution to the Ricci flow is trivial. We also consider the global behavior of the Type III warping product Ricci flow.


\end{abstract}

 \maketitle

 \section{Introduction}

 In the interesting work \cite{L}, B.List has introduced List's flow on a warped product space and proved that the finite time blow up of his flow (in case of the base manifold being closed) is an ancient solution of List flow (see Cor. 7.10 \cite{L}). Hence it is interesting to study the more properties of the ancient solution of Ricci flow, which is one topic of this note. The warped product spaces have their own geometric features, which are related to the concept of Ricci-Bakry-Emery curvature, see \cite{L}, \cite{C}, \cite{H}, and \cite{MZ}. In \cite{H}, global behavior of the Ricci flow on square torus bundle has been studied. In \cite {C}, through the estimate of positive lower bound of isoperimetric ratio on $S^2$, Cao has showed that the cigar times the real line can not be the finite blow-up limit of the Ricci flow on the warped product Ricci flow on $S^\times S^1$. In \cite{MZ}, Perelman's Lambda constant and eigenvalues of Laplacian operators between base spaces and whole spaces have been studied. One may refer to \cite{W} about the relation between mean curvature and Ricci-Bakry-Emery curvature on Riemannian measure spaces.

In the recent work \cite{LS}, J.Lott and N.Sesum have considered the three-dimensional warped product Ricci flows, which are also called the Ricci flows with symmetry and the flows differ from List's flows \cite{L} by Lie derivatives. Lott and Sesum can get the long time behavior of solutions of the 3-dimensional Ricci flows with symmetry. From their works, one can see that the Ricci flows with symmetry enjoy more better properties than general Ricci flows. We consider in this note two kinds of Ricci flows with symmetry, i.e., one is an ancient solution with symmetry of any dimension and the other is the 4-dimensional type III Ricci flow with symmetry.

In the first part of this short note, we answer affirmatively the question if an ancient solution of a warped product of a compact manifold with a line to the Ricci flow is trivial, see Theorem \ref{thm1} below for precise statement. Generally speaking, the proof may goes via the use of both Li-Yau type gradient estimate and Bernstein type estimate \cite{CH, CZ,K, M2, LY, Chow, Ni, M, LS, P02}. Recall that an ancient solution to the Ricci flow is a family of Riemannian metrics $(g(t)), -\infty<t<0$ on the manifold M such that
\begin{equation}\label{Ric-one}
\partial_fg=-2Rc(g), \ \ g=g(t),\ \ -\infty<t<0.
\end{equation}
In the second part of this note, we consider the global behavior of the Type III warping product Ricci flow. Note that there are many warping product spaces which are Ricci gradient Ricci solitons of the form $(R_+\times N,dr^2+f(r)^2h)$, where $(N,h)$ is a compact Einstein manifold and $f:R_+\to R$ is smooth function on half real line. For example, the Gaussian shrinking soliton on $R^n$ is of this form \cite{P02}. In this paper, we consider the warped product spaces of the form $(N\times R, h+e^{2u}ds^2)$, where $u:N\to R$ is a smooth function on the compact manifold $N$.

\section{Main results and proofs}

 Precisely, our question is if the warped product $M=N\times_u R$, where $u:N\times (-\infty,0)\to R$ is a smooth function on the compact manifold $N$, and $g(t)=h(t)+e^{2u(t)}ds^2$ with $h(t)$ being Riemannian metrics on $N$, can be an ancient solution to (\ref{Ric-one}). Let $n=dim N\geq 2$. We show the following Liouville type theorem.

\begin{Thm}\label{thm1} Assume that the warped product $M=N\times_u R$ with $g(t)=h(t)+e^{2u(t)}ds^2$ is an ancient solution to the Ricci flow (\ref{Ric-one}). If $N$ is compact, then the warped function $u$ is a constant.
 That is, $M$ is a direct product with $(N,h(t))$ being an ancient solution to Ricci flow.
 \end{Thm}

\begin{proof}
Assume that the warped product $M=N\times_u R$ with $g(t)=h(t)+e^{2u(t)}ds^2$ and $N$ compact is an ancient solution to the Ricci flow.
In this case, it is well-known that the Ricci flow equation is equivalent to the the List flow
\begin{equation}\label{Ric-2}
\partial_t h=-2Rc(h)+2\nabla u\bigotimes \nabla u,
\end{equation}
\begin{equation}\label{heat}
\partial_t u=\Delta_hu.
\end{equation}

In below, we use the flow (\ref{Ric-2}-\ref{heat}).
By Lemma 3.2 in \cite{L}, we know that
\begin{equation}\label{grad-1}
(\partial_t-\Delta_h)|\nabla u|^2=-2|D^2u|^2-2|\nabla u|^4.
\end{equation}

By applying the maximum principle to (\ref{grad-1}), we have
$$
\sup_N|\nabla u|^{2}(t)\leq \frac{1}{2(t-\alpha)}.
$$
Let $\alpha\rightarrow -\infty$ and we have $|\nabla u|=0$ on $M\times (-\infty,0)$, which implies that $u$ depends only on time variable and by the heat equation (\ref{heat}) we know that $u_t=0$.
Hence, $u$ is a constant on $M\times (-\infty,0)$.

This completes the proof of Theorem \ref{thm1}.

\end{proof}

 We remark that in the case when $N$ is complete non-compact, we need to assume that the curvature of $(N,h(t))$ is locally bounded and  $|\nabla u|^2$ is locally bounded and $|\nabla u|^2(x,t)\to 0$ as $|x|\to \infty$ for any $t\in (-\infty,0)$ so that the maximum principle can be used.

In the remaining part of this note, we consider the 4-dimensional type III Ricci flow $(M,g(t))$ of the warping spaces with positive scalar curvature. By Type III, we means a global Ricci flow $(g(t))$, $0<t<\infty$, with the curvature bound
$$
\sup_{M\times (0,\infty)} t|Rm|(g(t))=A<\infty.
$$

 Note that
$$
(\partial_t-\Delta_h)u^2=-2|\nabla u|^2.
$$
We have for any $\alpha\in (-\infty,0)$,
$$
(\partial_t-\Delta_h) [(t-\alpha)|\nabla u|^2+u^2]\leq 0.
$$
Fix $\alpha>0$, by using the maximum principle that
$$
(t-\alpha)|\nabla u|^2(t)+u^2(t)\leq C^2, \ \ for \ \ t>\alpha,
$$
which implies that
$$
|\nabla u|^2(t)\leq \frac{C^2}{t-\alpha}\to 0
$$
as $t \to\infty$.
As in \cite{L}, we let
$$
S_{ij}=R_{ij}-u_iu_j
$$
where $R_{ij}$ is the Ricci tensor of the metric $h$
and the scalar of $g(t)$ is
$$
S=R-|\nabla u|^2,
$$
where $R$ is the scalar curvature of $N$.

By (3.8) in Lemma 3.2 in \cite{L}, we know that there some uniform constant $C>0$ such that
\begin{equation}\label{second}
(\partial_t-\Delta_h)|D^2u|^2\leq -2|\nabla^3u|^2+C(|Rm(h)|+|\nabla u|^2)|D^2u|^2.
\end{equation}

Since $N$ has dimension three, we only need to bound the Ricci curvature $(R_ij)$. Using type III curvature bound and
 the Bernstein estimate about $|\nabla u|^2$ to the relation
 $$
R_{ij}=S_{ij}+u_iu_j,
$$
we know that $t|Ric(h(t))|\leq C$ for some uniform constant $C>0$ and for all $t\geq \alpha$. Hence $t|Rm(h(t))|$ is uniformly bounded for $t\geq \alpha$.

Let $\mu>0$ large and define
$$
F=|D^2u|^2+\mu |\nabla u|^2.
$$
Then by (\ref{grad-1}) and (\ref{second}) we have
$$
(\partial_t-\Delta_h)F\leq [C(|Rm(h)|+|\nabla u|^2)-\mu]F \leq 0.
$$
Applying the maximum principle we get that
$$
F(t)=|D^2u|^2(t)+\mu |\nabla u|^2(t)\leq C, \ \  t>\alpha
$$
for some uniform constant $C>0$, which depends only on the time $\alpha$.

Let
$$
F_1=(t-\alpha)|D^2u|^2+\mu |\nabla u|^2.
$$
Applying the Bernstein estimate and the estimate of $t|Rm(h)|$, we have for large $\mu>0$,
$$
(\partial_t-\Delta_h)F_1\leq [C(t-\alpha)(|Rm(h)|+|\nabla u|^2)-\mu]\leq 0.
$$
Applying the maximum principle, we have
$$
F_1(t)=(t-\alpha)|D^2u|^2+\mu |\nabla u|^2\leq C
$$
for some uniform constant $C$.

In conclusion we have proved the following result.

\begin{Thm}\label{thm2} Assume that the 4-dimensional warped product $M=N\times_u R$ with $g(t)=h(t)+e^{2u(t)}ds^2$ is type III  solution to the Ricci flow (\ref{Ric-one}). If $N$ is compact with $\alpha>0$, then we have for some uniform constant $C>0$
$$
t[|Rm(h(t))|+|D^2u(t)|^2+|\nabla u(t)|^2]\leq C, \ \  t>\alpha.
$$
 \end{Thm}

We remark that our result above can also been generalized to complete non-compact Riemannian manifolds.

\end{document}